\documentclass[11pt]{article} 
\usepackage{enumitem}
\setlist[enumerate]{label=(\roman*)} 

\usepackage{tabularx} 
\usepackage{amssymb}
\usepackage{amsmath}
\usepackage{amsthm}
\usepackage{latexsym}
\usepackage{amsfonts}
\usepackage{geometry}
\usepackage{mathtools}

\usepackage{bbm} 

\newtheorem{Th}{Theorem}
\theoremstyle{definition}
\newtheorem{Lem}[Th]{Lemma}
\newtheorem{Def}[Th]{Definition}

\newtheorem{Exa}[Th]{Example}

\newtheorem{Rem}[Th]{Remark}

\newtheorem{Ass}[Th]{Assumption}
\newtheorem*{Th*}{Theorem}
\newtheorem*{Lem*}{Lemma}
\newtheorem*{Def*}{Definition}
\newtheorem*{Not*}{Notation}
\newtheorem*{Fac*}{Fact}
\newtheorem*{Exa*}{Example}
\newtheorem*{Obs*}{Observation}
\newtheorem*{Cor*}{Corollary}
\newtheorem*{Rem*}{Remark}
\newtheorem*{Pro*}{Proposition}
\newtheorem*{Exe*}{Exercise}
\newtheorem*{Que*}{Question}
\newtheorem*{Prob*}{Problem}
\newtheorem*{Ass*}{Assumption}

\usepackage{mathtools}
\DeclarePairedDelimiter\abs{\lvert}{\rvert}%
\DeclarePairedDelimiter\norm{\lVert}{\rVert}%

\makeatletter
\let\oldabs\abs
\def\abs{\@ifstar{\oldabs}{\oldabs*}}
\let\oldnorm\norm
\def\norm{\@ifstar{\oldnorm}{\oldnorm*}}
\makeatother

\usepackage{mathrsfs}

\newcommand{\R}{\mathbb{R}}

\newcommand{\N}{\mathbb{N}}

\newcommand{\ud}{\mathrm{d}}

\newcommand{\ds}{\, \mathrm{d}s}
\newcommand{\dt}{\, \mathrm{d}t}

\newcommand{\dL}{\, \mathrm{d}L}

\newcommand{\dX}{\, \mathrm{d}X}

\newcommand{\E}{\mathbb{E}}

\newcommand{\1}{\mathbbm{1}} 

\let\phi\varphi
\let\epsilon\varepsilon

\DeclareMathOperator{\Id}{\mathrm{Id}}

\newcommand\Item[1][]{%
  \ifx\relax#1\relax  \item \else \item[#1] \fi
  \abovedisplayskip=0pt\abovedisplayshortskip=0pt~\vspace*{-\baselineskip}}

\allowdisplaybreaks

\usepackage{layouts}

\usepackage{tikz}
\tikzset{node distance=2cm, auto}
\usetikzlibrary{matrix}

\usepackage{tikz-cd}

\tikzset{
  symbol/.style={
    draw=none,
    every to/.append style={
      edge node={node [sloped, allow upside down, auto=false]{$#1$}}}
  }
}

\title{Stochastic integration with respect to \\ cylindrical L\'evy processes\\
by $p$-summing operators}

\date{9 December 2019}

\author{Tomasz Kosmala \thanks{tomasz.kosmala@kcl.ac.uk} \qquad\qquad\qquad  Markus Riedle \thanks{markus.riedle@kcl.ac.uk} \\Department of Mathematics \\ King's College London\\ London WC2R 2LS\\ United Kingdom}

\begin{document}

\maketitle

\begin{abstract}
We introduce a stochastic integral with respect to cylindrical L\'evy processes with finite $p$-th weak moment for $p\in[1,2]$. The space of integrands consists of $p$-summing operators between Banach spaces of martingale type $p$. We apply the developed integration theory to establish the existence of a solution for a stochastic evolution equation driven by a cylindrical L\'evy process.
\end{abstract}




\section{Introduction}
Cylindrical L\'evy processes are a natural generalisation of cylindrical Brownian motions to the non-Gaussian setting, and they can serve as a model of random perturbation of partial differential equations or other dynamical systems. As a generalised random process, cylindrical L\'evy processes do not attain values in the underlying space, and they do not enjoy a L\'evy-It{\^o} decomposition in general. Since conventional approaches to stochastic integration rely on either stopping time arguments or a semi-martingale decomposition in the one or other form, a completely novel method for stochastic integration has been introduced in the work \cite{Jakubowski_Riedle} by one of us with Jakubowski.  This method is purely based on tightness arguments, since typical estimates of the expectation are not available without utilising stopping time arguments or a semi-martingale decomposition. As a consequence, although this method guarantees the existence of the stochastic integral for a large class of random integrands, it does not provide any control of the integrals. Since many applications, such as modelling dynamical systems or control problems, require upper estimates of the stochastic integrals, this method seems to be difficult to use for such applications. 

In order to provide a control of the stochastic integral, we develop a theory of stochastic integration for random integrands with respect to  cylindrical L\'evy processes with finite $p$-th weak moments for $p\in [1,2]$ in this work. Our approach enables us to develop the theory on a large class of general Banach spaces. We apply the obtained estimates to establish the existence of an abstract partial differential equations driven by a cylindrical L\'evy process with   finite $p$-th weak moments. 

Stochastic integration with respect to a cylindrical Wiener process is well  developed in Hilbert spaces and various classes of Banach spaces. Typical Banach spaces which permit a development of stochastic integration are martingale type 2 Banach spaces, see e.g.\ Dettweiler \cite{Dettweiler_1983, Dettweiler} or UMD spaces, see e.g.\ van Neerven, Veraar and Weis \cite{Neerven_Veraar_Weis-integration}.  Veraar and Yaroslavtsev
\cite{Veraar_Yaroslavtsev} extend the approach for UMD spaces in \cite{Neerven_Veraar_Weis-integration}
to continuous cylindrical  local martingales by utilising the Dambis-Dubins-Schwarz Theorem.  Stochastic integration in Hilbert spaces with respect to genuine L\'evy processes is for example presented by Peszat and Zabczyk in \cite{Peszat_Zabczyk}, and with respect to cylindrical L\'evy processes the theory is developed in \cite{Jakubowski_Riedle}.  Stochastic integration with respect to a Poisson random measure in Banach spaces is developed for example  by Mandrekar and R\"{u}diger in \cite{Mandrekar_Rudiger} and by Brze\'{z}niak, Zhu and Hausenblas \cite{Brzezniak_Hausenblas_Zhu}. 

In this work we are faced with the similar problem as in \cite{Jakubowski_Riedle}. Conventional approaches to stochastic integration utilise either stopping times or the L\'evy-It{\^o} decomposition to show continuity of the integral operator separately: firstly with respect to the martingale part with finite $2$-nd moments and secondly with respect to the bounded variation part. However, since these approaches are excluded for cylindrical L\'evy processes, we show continuity of the integral operator ``in one piece", i.e.\ without applying the semimartingale decomposition of the integrator. For this purpose, we utilise a generalised form of Pietsch's factorisation theorem, originating from the work of Schwartz \cite{Schwartz_seminar_book}. 

More specifically, the space of admissible integrands are predictable stochastic processes with values in the space of $p$-summing operators and with integrable $p$-summing norm for $p\in [1,2]$ in this work. Due to results by Kwapie{\'n} and Schwartz, for $p>1$ the space of $p$-summing operators coincides with the space of $p$-Radonifying operators, which are exactly the operators which map each cylindrical random variable with finite $p$-th weak moments to a genuine random variable. In this way, stochastic processes with values in the space of $p$-summing operators are the natural class of integrands, as they map the cylindrical increments of the integrator to the genuine random variables. Furthermore, the class of $p$-summing operators coincides with the class of Hilbert-Schmidt operators in Hilbert spaces, and as such the aforementioned space of admissible integrands is a natural generalisation of the integration theory in Hilbert spaces with respect to genuine L\'evy process in e.g.\ \cite{Peszat_Zabczyk}.
In typical applications such as the heat equation, the $p$-summing norm of the opeartors appearing in the equation can be explicitely estimated, see \cite{Brzezniak_Long}.

In Section \ref{sec_preliminaries}, we recall the concepts of cylindrical measures and cylindrical L\'evy processes. 
In Section \ref{sec_Radonification}, we present the generalised Pietsch's factorisation theorem due to Schwartz, and derive a result on the strong convergence of $p$-summing operators, which is needed in the proof of the stochastic continuity of the stochastic convolution. Section  \ref{sec_radonification_of_increments} is devoted to the construction of the stochastic integral. This is done in two steps as in the article \cite{Riedle_L2} by the second author. Firstly we Radonify the increments of the cylindrical L\'evy process by random $p$-summing operators. Secondly, we define the integral for simple integrands and extend it by continuity to the general ones. We also present some examples of the processes covered by our theory. In Section \ref{sec_existence} we apply our results to establish existence and uniqueness solution to the evolution equation driven by a cylindrical L\'evy noise with finite $p$-th weak moments for $p\in[1,2]$.

\section{Preliminaries}
\label{sec_preliminaries}

Let $E$ and $F$ be Banach spaces with separable duals $E^*$ and $F^*$.  The operator norm of an operator $u\colon E\to F$ is denoted with $\norm{u}_{\mathcal{L}(E,F)}$ or simply $\norm{u}$. We write $B_E$ for the closed unit ball in $E$. The Borel $\sigma$-field is denoted with $\mathcal{B}(E)$.

Fix a probability space $(\Omega,\mathcal{F},P)$ with a filtration $(\mathcal{F}_t)$.
We denote the space of equivalence classes of real-valued random variables equipped with the topology of convergence in probability by $L^0(\Omega,\mathcal{F},P;\R)$. The Bochner space of equivalence classes of $E$-valued, random variables with finite $p$-th moment is denoted with $L^p(\Omega,\mathcal{F},P;E)$. In case the codomain is not separable we take only separably valued random variables.

\emph{Cylindrical sets} are sets of the form
$$C(x_1^*,\ldots,x_n^*;B) = \{ x \in E : (x_1^*(x),\ldots x_n^*(x)) \in B \}$$
for $x_1^*,\ldots ,x_n^* \in E^*$ and $B \in \mathcal{B}(\R^n)$. For $\Gamma \subseteq E^*$ we denote with $\mathcal{Z}(E,\Gamma)$ the collection of all cylindrical sets with $x_1^*,\ldots, x_n^* \in \Gamma$, $B\in\mathcal{B}(\R^n)$ and $n\in\N$. If $\Gamma=E^*$, we write $\mathcal{Z}(E)$ to denote the collection of all cylindrical subsets of $E$. Note that $\mathcal{Z}(E)$ is an algebra and that if $\Gamma$ is finite, then $\mathcal{Z}(E,\Gamma)$ is a $\sigma$-algebra. A function $\mu \colon \mathcal{Z}(E) \to [0,\infty]$ is called a \emph{cylindrical measure} if its restriction to $\mathcal{Z}(E,\Gamma)$ is a measure for every finite subset $\Gamma \subseteq E^*$.
If $\mu(E)=1$ we call it a \emph{cylindrical probability measure}.
A \emph{cylindrical random variable} is a linear and continuous mapping
$$Y \colon E^* \to L^0(\Omega,\mathcal{F},P;\R).$$
The \emph{cylindrical distribution} of a cylindrical random variable $Y$ is defined by
$$\mu(C(x_1^*,\ldots,x_n^*;B)) = P((Yx_1^*,\ldots,Yx_n^*) \in B),$$
which defines a cylindrical probability measure on $\mathcal{Z}(E)$. 
The characteristic function of a cylindrical random variable (resp.\ cylindrical probability measure) is given by
$$\phi_{Y}(x^*) = \E \left[ e^{iYx^*} \right], \qquad \left(\text{resp.\ } \phi_\mu(x^*) = \int_E e^{ix^*(x)} \, \mu(\ud x)\right),$$
for $x^* \in  E^*$. We say that a cylindrical random variable $Y$ (resp.\ cylindrical measure $\mu$) is of \emph{weak order $p$} or has {\em finite $p$-th weak moments} if $\E \left[ \abs{Yx^*}^p \right] < \infty$ for every $x^* \in E^*$, (resp.\ $\int_E \abs{x^*(x)}^p \mu(\ud x)< \infty$).
We say that $Y$ is \emph{induced} by an $E$-valued random variable $X\colon\Omega\to E$ if
$$Yx^* = x^*(X) \qquad \text{ for all } x^* \in E^*.$$

A family of cylindrical random variables $(L(t) : t \geq 0)$ is called a \emph{cylindrical L\'evy process} if, for every $x_1^*,\ldots ,x_n^* \in E^*$ and $n\in \N$, we have that
$$\big( (L(t)x_1^*,\ldots, L(t)x_n^*) : t\geq 0 \big)$$
is a L\'evy process in $\R^n$ with respect to the filtration $(\mathcal{F}_t)$.
We say that $L$ is \emph{weakly $p$-integrable} if $\E \left[ \abs{L(1)x^*}^p\right]<\infty$ for every $x^* \in E^*$.
The characteristic function of $L(1)$ can be written in the form
\begin{equation*}
\phi_{L(1)}(x^*)
= \exp \left( i p(x^*) - \tfrac12 q(x^*) + \int_E \left( e^{i x^*(x)} -1 -ix^*(x)\1_{B_\R}(x^*(x)) \right) \nu(\mathrm{d}x) \right),
\end{equation*}
where $p\colon E^* \to \R$ is a continuous function with $p(0)=0$, $q\colon E^* \to \R$ is a quadratic form, and $\nu$ is a finitely additive set function on cylindrical sets of the form $C(x_1^\ast, \cdots, x_n^\ast;B)$ for 
$x_1^\ast,\dots, x_n^\ast\in E^\ast$ and $B\in\mathcal{B}(\R^n\setminus \!\{0\})$,  such that for every $x^* \in E^*$ it satisfies
$$\int_{\R\setminus\{0\}} \left( \abs{\beta}^2 \wedge 1 \right) \,(\nu\circ (x^{\ast})^{-1})(\ud x) < \infty.$$
Cylindrical L\'evy processes are introduced in \cite{Applebaum_Riedle} and further details on the characteristic function can be found in \cite{Riedle_infinitely}.

An operator $u\colon E\to F$ is called $p$-summing if there exists a constant $c$ such that
\begin{equation}
\label{p-summing_definition}
\left(\sum_{k=1}^n \norm{u(x_k)}^p \right)^{1/p} \leq c \sup\left\{ \left( \sum_{k=1}^n \abs{x^*(x_k)}^p \right)^{1/p} : x^* \in B_{E^*}\right\}
\end{equation}
for all $x_1,\ldots ,x_n \in E$ and $n\in \N$; see \cite{Diestel}.  We denote with $\pi_p(u)$ its $p$-summing norm, which is the smallest possible constant $c$ in \eqref{p-summing_definition}. The space of $p$-summing operators is denoted with $\Pi_p(E,F)$. If $E$ and $F$ are Hilbert spaces, this space coincides with the space of Hilbert-Schmidt operators denoted by $L_{\rm HS}(E,F)$ with the norm $\norm{\cdot}_{L_{\rm HS}(E,F)}$; see \cite[Th.\ 4.10 and Cor.\ 4.13]{Diestel}. Moreover, the $p$-summing norms and the Hilbert-Schmidt norm in $L_{\rm HS}(E,F)$ are equivalent. 

A Banach space $E$ is of \emph{martingale type $p\in [1,2]$} if there exists a constant $C_p$ such that for all finite $E$-valued martingales $(M_k)_{k=1}^n$ the following inequality is satisfied:
\begin{equation}
\label{martingale_type_p}
\sup_{k=1,\ldots,n} \E \big[ \norm{M_k}^p \big] \leq C_p \sum_{k=1}^n \E \big[ \norm{M_k-M_{k-1}}^p \big],
\end{equation}
where we use the convention that $M_{0}=0$; see \cite{Hytonen_Neerven_Veraar_Weis}.

We use the notation $u(\mu)$ for the push forward cylindrical measure $\mu \circ u^{-1}$ for a continuous linear function $u\colon E \to F$ and a cylindrical measure $\mu$. An operator $u \colon E \to F$ is called $p$\emph{-Radonifying} for some $p\ge 0$ if for every cylindrical measure $\mu$ on $E$ of weak order $p$, the measure $u(\mu)$ extends to a Radon measure on $F$ with finite $p$-th strong moment. 
Equivalently, the mapping $u$ is $p$-Radonifying if for every cylindrical random variable $Y$ on $E^*$ with finite weak $p$-th moment, the cylindrical random variable $Y \circ u^*$ is induced by an $F$-valued random variable with finite $p$-th strong moment; see \cite[Prop.\ VI.5.2]{Vakhania}.

A Banach space $E$ has the \emph{approximation property} if for every compact set $K \subseteq E$ and for every $\epsilon>0$ there exists a finite rank operator $u \colon E \to E$ such that $\norm{u(x)-x} \leq \epsilon $ for $x \in K$. 

A Banach space $E$ has the Radon-Nikodym property if for any probability space $(\Omega,\mathcal{F},P)$ and vector-valued measure $\mu \colon \mathcal{F} \to E$, which is absolutely continuous with respect to $P$, there exists $f\in L^1(\Omega,\mathcal{F},P;E)$ such that
$$\mu(A) = \int_A f(\omega) \, P(\ud \omega) \qquad\text{for all } A \in \mathcal{F}.$$
It is well known that every reflexive Banach space has the Radon-Nikodym property; see \cite[Cor.\ 2, p.\ 219]{Vakhania}.
It follows from \cite[Th.\ VI.5.4 and Th.\ VI.5.5]{Vakhania} that if either $p>1$ or $p=1$ and $F$ has the Radon-Nikodym property, then the classes of $p$-Radonifying and $p$-summing operators between $E$ and $F$ coincide.

\section{Some results on $p$-suming operators}
\label{sec_Radonification}

Our approach to stochastic integration with respect to a cylindrical L\'evy process is based on 
a generalisation of Pietsch's factorisation theorem, which is due to Schwartz; see \cite[p.\ 23-28]{Schwartz} and \cite{Schwartz_III}. For a measure $\mu$ on $\mathcal{B}(E)$ and $p\in [1,2]$ we define
$$\norm{\mu}_p := \left( \int_E \norm{x}^p \mu(dx)\right)^{1/p},$$
and say that  $\mu$ is  of order $p$ if $\norm{\mu}_p<\infty$.
For a cylindrical measure $\mu$ on $\mathcal{Z}(E)$ we define
$$\norm{\mu}_p^* = \sup_{x^* \in B_{E^*}} \norm{x^*(\mu)}_p,$$
and we say that $\mu$ is of weak order $p$ if $\norm{\mu}_p^\ast<\infty$.

\begin{Th}
\label{th_Schwartz}
For $p\in [1,2]$, assume either that $p>1$ or that $F$ has the Radon-Nikodym property if $p=1$. Each 
 cylindrical probability measure $\mu$ on $\mathcal{Z}(E)$ and $p$-summing map $u \colon E \to F$ satisfy
\begin{equation}
\label{crucial_inequality}
\norm{u(\mu)}_p \leq \pi_p(u) \norm{\mu}_p^*.
\end{equation}
\end{Th}
\begin{proof}
See \cite{Schwartz} or \cite{Schwartz_III,Schwartz_seminar_book}.
\end{proof}
\begin{Rem}
Pietsch's factorisation theorem states that if $u\colon E\to F$ is a $p$-summing map then there exists a regular probability measure $\rho$ on $B_{E^\ast}$ such that 
\begin{align*}
\norm{ux}\le \pi_p(u)\left(\int_{B_{E^\ast}}\abs{x^\ast(x)}^p\, \rho(\ud x)\right)^{1/p}
\qquad\text{for all }x\in E. 
\end{align*}
If $X$ is a genuine random variable $X\colon\Omega\to E$ with probability distribution $\mu$ on $\mathcal{B}(E)$,
Pietsch's factorisation theorem immediately implies 
\begin{align*}
\norm{u(\mu)}_p^p=
\E \left[ \norm{uX}^p\right] 
\le \left(\pi_p(u)\right)^p \E\left[ \int_{B_{E^\ast}}\abs{x^\ast(X)}^p\, \rho(\ud x)\right]
\le  \left(\pi_p(u)\right)^p \norm{\mu}_p^p.
\end{align*}
For this reason, we refer to Theorem \ref{th_Schwartz} as generalised Pietsch's factorisation theorem. 
\end{Rem}

For establishing the stochastic continuity of the stochastic convolution in Section \ref{sec_existence}, we need a result on the convergence of $p$-summing operators between Banach spaces. In the case of Hilbert spaces, this convergence result can easily be seen: 
suppose that $U$ and $H$ are separable Hilbert spaces and let $\psi \colon U \to H$ be a Hilbert-Schmidt operator. If $(\phi_n)$ is a sequence of operators $\phi_n \colon H \to H$  converging strongly to $0$ as $n \to \infty$,  then the composition $\phi_n \psi$ converges to $0$ in the Hilbert-Schmidt norm.
Indeed, take $(e_n)$ an orthonormal basis of $U$ and calculate
$$\norm{\phi_n \psi}_{L_{\rm HS}(U,H)}^2 = \sum_{k=1}^\infty \norm{\phi_n \psi e_k}^2.$$
Every term in the above sum converges to $0$ as $n \to \infty$ due to  the strong convergence of $\phi_n$.  By Lebesgue's dominated convergence theorem we obtain $\norm{\phi_n \psi}_{L_{\rm HS}(U,H)}^2 \to 0$.
The following result extends this conclusion in Hilbert spaces to the Banach space setting by approximating  $p$-summing operators with finite rank operators.

\begin{Th}
\label{th_convergence_in_p_summing_norm}
Suppose that $E$ is a reflexive Banach space or a Banach space with separable dual and that $E^{**}$ has the approximation property.
If $\psi \colon E\to F$ is a $p$-summing operator and $(\phi_n)$ is a sequence of operators $\phi_n \colon F\to F$ converging strongly to $0$ then we have 
\begin{equation}
\label{convergence_in_p_summing_norm}
\pi_p(\phi_n \psi) \to 0.
\end{equation}
\end{Th}

\begin{proof}
We first prove the assertion for finite rank operators  $\psi \colon E\to F$, in which case we can assume that $\psi = \sum\limits_{k=1}^N x_k^* \otimes y_k$ for some $x_k^\ast \in E^\ast$ and $y_k\in F$. Then $\phi_n \psi = \sum\limits_{k=1}^N x_k^* \otimes (\phi_n y_k)$ and since $\pi_p(x^*\otimes y)=\norm{x^*}\norm{y}$
by a simple argument (see \cite[p.\ 37]{Diestel}), 
 we estimate
$$\pi_p(\phi_n \psi) \leq \sum_{k=1}^N \pi_p(x_k^* \otimes (\phi_n y_k)) = \sum_{k=1}^N \norm{x_k^*} \norm{\phi_n y_k} \to 0, $$
because $\norm{\phi_n y_k} \to 0$ for every $k\in\{1,\dots, N\}$.

Consider now the case of a general $p$-summing operator $\psi$. Under the assumptions on $E$ and $F$, by Corollary 1 in \cite{Saphar}, the finite rank operators are dense in the space of $p$-summing operators. That is, there exists a sequence of finite rank operators $(\psi_k)$ such that $\pi_p(\psi_k-\psi) \to 0$ as $k\to \infty$. It follows that
\begin{equation}
\label{from_triangle_inequality}
\pi_p(\phi_n \psi) \leq \pi_p(\phi_n\psi_k) + \pi_p(\phi_n(\psi-\psi_k))\qquad
\text{for all }k,\, n\in\N.
\end{equation}
Fix $\epsilon>0$ and let $c:=\sup\{\norm{\phi_n}:\, n\in\N\}$.
Choose $k\in \N$ such that $\pi_p(\psi-\psi_k) \leq \frac{\epsilon}{2c}$.
Since $\psi_k$ is a finite rank operator, the argument above guarantees that there exists $n_0\in \N$ such that for all $n\geq n_0$ we have $\pi_p(\phi_n\psi_k) \leq \frac{\epsilon}{2}$. 
Inequality \eqref{from_triangle_inequality} implies for all $n\geq n_0$ that  
$\pi_p(\phi_n \psi) \leq \tfrac{\epsilon}{2} + \tfrac{\epsilon}{2} $, which completes the proof.
\end{proof}

\begin{Rem}
The proof of Theorem \ref{th_convergence_in_p_summing_norm} relies on the density of finite rank operators in the space of $p$-summing operators. This holds under more general assumption than assumed in Theorem \ref{th_convergence_in_p_summing_norm}; see  \cite[p.\ 384 and 388]{Saphar}. 

However, the result of Theorem \ref{th_convergence_in_p_summing_norm}  does not hold in the case of arbitrary Banach spaces as the following example shows. 
Choose  $E=\ell^1(\R)$ and $F=\ell^2(\R)$ and equip both spaces with the canonical basis $(e_n)$, where $e_n=(0,\ldots,0,1,0,\ldots)$. We take $\psi = \Id \colon E\to F$, which  is $1$-Radonifying by Grothendieck's Theorem; see \cite[p.\ 38-39]{Diestel}. Furthermore, we define 
 $\phi_n = e_n \otimes e_n$,  i.e.\ $\phi_n(x) =x(n)e_n=(0,\ldots,0,x(n),0,\ldots)$ for a sequence $x=(x(n))\in \ell^2(\R)$. Then $\phi_n$ converges to $0$ strongly as $n\to \infty$, but since $\phi_n \psi$ is finite rank we have $\pi_1(\phi_n \psi)=\norm{e_n}\norm{e_n}=1$ for all $n\in\N$. This counterexample shows that the assumptions on the space $E$ in Theorem \ref{th_convergence_in_p_summing_norm} cannot be dropped.
\end{Rem}


\section{Radonification of increments and Stochastic integral}
\label{sec_radonification_of_increments}

In this section we fix $p\in [1,2]$ and assume that the cylindrical L\'evy process $L$ has finite $p$-th weak moments
and  assume either that $p>1$ or that $F$ has the Radon-Nikodym property if $p=1$. 

Fix $0\le s< t\le T$. An  $\mathcal{F}_s$-measurable random variable $\Psi  \colon  \Omega \to \Pi_p(E,F)$ is called \emph{simple} if it is of the form
\begin{equation}
\label{eq_simple_random_variable2}
\Psi = \sum_{k=1}^m \1_{A_k} \psi_k,
\end{equation}
for some disjoint sets $A_1,\ldots ,A_m \in \mathcal{F}_s$ and $\psi_1,\ldots ,\psi_m \in \Pi_p(E,F)$.
The space of simple, $\mathcal{F}_s$-measurable random variables is denoted with $S:=S(\Omega,\mathcal{F}_s;\Pi_p(E,F))$, and it is equipped with the norm $\norm{\Psi}_{S,p} := \left( \E \left[ \pi_p(\Psi)^p \right] \right)^{1/p}$.

Since for $p>1$ or for $p=1$ with $F$ having the Radon-Nikodym property, each $p$-summing operator $\psi_k\colon E\to F$ is $p$-Radonifying, 
it follows that the cylindrical random variable $(L(t)-L(s))\psi_k^*$ is induced by a classical, $F$-valued random variable $X_{\psi_k}\colon \Omega\to F$: 
$$\big( L(t)-L(s) \big) \psi_k^* (x^*) = x^*(X_{\psi_k}) \qquad \text{ for all } x^* \in F^*.$$
This enables us to define the operator 
\begin{align}\label{eq.def-J}
J_{s,t} \colon S(\Omega,\mathcal{F}_s;\Pi_p(E,F)) \to L^p(\Omega, \mathcal{F},P;F), \qquad J_{s,t}(\Psi) := \sum_{k=1}^m \1_{A_k} X_{\psi_k}.
\end{align}
The following result allows us to extend the operator $J_{s,t}$ to $L^p(\Omega,\mathcal{F}_s,P;\Pi_p(E,F))$.
\begin{Lem}
\label{lem_construction_operators_J} (Radonification of the increments)\\
For fixed $0\le s <t\le T$, the operator $J_{s,t}$ defined in \eqref{eq.def-J} is continuous and satisfies
\begin{equation}
\label{norm_of_J_s_t}
\norm{J_{s,t}}_{\mathcal{L}(S,L^p)} \leq \norm{L(t-s)}_{\mathcal{L}(E^*,L^p(\Omega;\R))}, 
\end{equation}
and thus $J_{s,t}$ can be extended to $J_{s,t} \colon L^p(\Omega,\mathcal{F}_s,P;\Pi_p(E,F)) \to L^p(\Omega, \mathcal{F}_t, P; F)$. 
\end{Lem}

\begin{proof}
Let $\Psi$ be of the form \eqref{eq_simple_random_variable2}. 
Since the sets $A_k$ are disjoint it follows that 
$$\E \left[ \norm{J_{s,t}(\Psi)}^p \right] 
= \E \left[ \norm{\sum_{k=1}^m \1_{A_k} X_{\psi_k}}^p \right]
= \E \left[ \sum_{k=1}^m \1_{A_k} \norm{X_{\psi_k}}^p \right].$$
Using the fact that each $A_k$ is $\mathcal{F}_s$-measurable and that $X_{\psi_k}$ is independent of $\mathcal{F}_s$ we can calculate further 
\begin{equation}
\label{conditional_expectation_calculation}
\E \big[ \norm{J_{s,t}(\Psi)}^p \big]
=  \sum_{k=1}^m \E \big[ \E \big[ \1_{A_k} \norm{X_{\psi_k}}^p \vert \mathcal{F}_s \big] \big] 
=  \sum_{k=1}^m P(A_k)\E \big[ \norm{X_{\psi_k}}^p   \big]. 
\end{equation}
In order to estimate $\E \left[ \norm{X_{\psi_k}}^p \right]$ we apply Theorem \ref{th_Schwartz} to obtain that
\begin{align} \label{eq.Schwartz-applied}
\left( \E \left[ \norm{X_{\psi_k}}^p \right] \right)^{1/p}
\leq \pi_p(\psi_k) \norm{L(t)-L(s)}_p^*.
\end{align}
Since stationary increments of the real-valued L\'evy processes yield 
$$\left( \E \left[ \abs{(L(t)-L(s))x^*}^p\right] \right)^{1/p}
=\left( \E \left[ \abs{L(t-s)x^*}^p\right] \right)^{1/p} \qquad\text{for all }x^*\in E^*, 
$$
it follows that 
\begin{equation}
\label{cylindrical_p_norm_of_increment}
\norm{L(t)-L(s)}_p^*
= \sup_{x^*\in B_{E^*}}  \left( \E \left[ \abs{L(t-s)x^*}^p\right] \right)^{1/p}
=\norm{L(t-s)}_{\mathcal{L}(E^*,L^p(\Omega;\R))}.
\end{equation}
Note, that by the closed graph theorem and the continuity of $L(t-s) \colon E^* \to L^0(\Omega,\mathcal{F},P;\R)$, the mapping $L(t-s) \colon E^* \to L^p(\Omega,\mathcal{F},P;\R)$ is continuous.
This shows that the last expression in \eqref{cylindrical_p_norm_of_increment} is finite.
Applying estimates \eqref{eq.Schwartz-applied} and \eqref{cylindrical_p_norm_of_increment} to
\eqref{conditional_expectation_calculation} results in 
\begin{equation}
\label{expectation_of_J_s_t}
\begin{aligned}
\left(\E \left[ \norm{J_{s,t}(\Psi)}^p \right] \right)^{1/p}
&\leq \left( \sum_{k=1}^m P(A_k) \pi_p(\psi_k)^p \norm{L(t-s)}_{\mathcal{L}(E^*,L^p(\Omega;\R))}^p \right)^{1/p} \\
&= \norm{L(t-s)}_{\mathcal{L}(E^*,L^p(\Omega;\R))} \left(\E \big[ \pi_p(\Psi)^p \big]\right)^{1/p},
\end{aligned}
\end{equation}
which proves \eqref{norm_of_J_s_t}.
\end{proof}

For defining the stochastic integral below, let $\Lambda(\Pi_p(E,F))$ denote the space of predictable processes $\Psi\colon [0,T]\times\Omega \to \Pi_p(E,F)$ such that
$$\norm{\Psi}_\Lambda
:= \left( \E \left[ \int_0^T \pi_p(\Psi(s))^p \,\ds \right] \right)^{1/p}
<\infty,$$
that is $\Lambda(\Pi_p(E,F)) = L^p\big([0,T] \times \Omega, \mathcal{P}, \ud t \otimes P; \Pi_p(E,F)\big)$, where $\mathcal{P}$ denotes the predictable $\sigma$-algebra on $[0,T] \times \Omega$.
A simple stochastic process is of the form
\begin{equation}
\label{simple_process}
\Psi(t) = \Psi_0\1_{\{0\}}(t)+ \sum_{k=1}^{N-1} \Psi_k \1_{(t_k,t_{k+1}]}(t),
\end{equation}
where $0=t_1<\cdots<t_N=T$, and each $\Psi_k$ is an $\mathcal{F}_{t_k}$-measurable,  $\Pi_p(E,F)$-valued random variable with $E[ \pi_p(\Psi_k)^p]<\infty$. 
We denote with $\Lambda_0^S(\Pi_p(E,F))$ the space of simple processes of the form \eqref{simple_process} where each $\Psi_k$ is a simple random variable of the form \eqref{eq_simple_random_variable2}, i.e.\ taking only a finite number of values.

Since for stochastic processes in $\Lambda_0^S(\Pi_p(E,F))$ the Radonification of the increments 
are defined by the operator $J_{s,t}$, we can define the integral operator by
\begin{align}\label{eq.define-I}
I\colon \Lambda_0^S(\Pi_p(E,F)) \to L^p(\Omega,\mathcal{F}_T,P;F),\qquad
I(\Psi) := \sum_{k=1}^{N-1} J_{t_k,t_{k+1}}(\Psi_k), 
\end{align}
where $\Psi$ is assumed to be of the form \eqref{simple_process}.

\begin{Lem}
\label{Lem_approximation_by_double_simple}
The space $\Lambda_0^S(\Pi_p(E,F))$ is dense in $\Lambda(\Pi_p(E,F))$
w.r.t.\ $\norm{\cdot}_{\Lambda}$.
\end{Lem}

\begin{proof}
The result follows from the construction in the proof of \cite[Prop.\ 4.22(ii)]{Da_Prato_Zabczyk}.
\end{proof}

\begin{Th} (stochastic integration)\\
\label{th_continuity_of_integral}
Assume that the cylindrical L\'evy process $L$ has the characteristics $(b,0,\nu)$ 
and satisfies 
\begin{equation}
\label{finiteness_p_integral}
\int_E \abs{x^*(x)}^p \, \nu(\ud x) < \infty, \qquad \text{ for all } x^* \in E^*
\end{equation}
and that $F$ is of martingale type $p$ and has the Radon-Nikodym property if $p=1$. 
Then the integral operator $I$ defined in \eqref{eq.define-I} is continuous and extends to the operator
$$I \colon \Lambda(\Pi_p(E,F)) \to L^p(\Omega,\mathcal{F}_T,P;F).$$
\end{Th}

\begin{proof}
 Let $\Psi$ in $\Lambda_0^S(\Pi_p(E,F))$ be given by \eqref{simple_process} where 
 $\Psi_k$  is of the form
$$\Psi_k = \sum_{i=1}^{m_k} 1_{A_{k,i}} \psi_{k,i},$$
for some disjoint sets $A_{k,1},\ldots ,A_{k,m_k} \in \mathcal{F}_{t_k}$ and $\psi_{k,1},\ldots ,\psi_{k,m_k} \in \Pi_p(E,F)$ for all $k\in\{0,\dots, N-1\}$.

The cylindrical L\'evy process $L$ can be decomposed into $L(t)x^* = B(t)x^* + M(t)x^*$
for all $x^\ast \in E^\ast$, where  $B(t)x^* := t\, \E \left[ L(1)x^* \right]$ and $M(t)x^* := L(t)x^*-B(t)x^* $ for all $x^\ast\in E^\ast$ and $t\ge 0$. Both $B(t)\colon E^\ast \to L^1(\Omega,\mathcal{F},P;\R)$ and $M(t)\colon E^\ast \to L^1(\Omega,\mathcal{F},P;\R)$ are linear and continuous since $L(1) \colon E^* \to L^1(\Omega,\mathcal{F},P;\R)$ is continuous due to the closed graph theorem. In particular, both $B$ and $M$ are cylindrical L\'evy processes, and we can integrate separately with respect to $B$ and $M$:
\begin{equation}
\label{decomposition_of_integral}
I(\Psi) = I_B(\Psi) + I_M(\Psi).
\end{equation}
For the first integral in \eqref{decomposition_of_integral} we calculate
\begin{align*}
\norm{I_B(\Psi)}^p 
= \sup_{y^* \in B_{F^*}} \abs{y^*\left( \int_0^T \Psi(s) \, \ud B(s) \right)}^p 
= \sup_{y^* \in B_{F^*}} \abs{ \int_0^T B(1)(\Psi^*(s)y^*)\ds}^p.
\end{align*}
By H\"older's inequality with $q=\frac{p}{p-1}$ and $q=\infty$ if $p=1$ we obtain
\begin{equation*}
\begin{aligned}
\norm{I_B(\Psi)}^p
& \leq  \sup_{y^* \in B_{F^*}} T^{p/q} \int_0^T \abs{B(1)(\Psi^*(s)y^*)}^p \ds  \\
&\leq T^{p/q}\norm{B(1)}_{\mathcal{L}(E^*,\R)}^p  \int_0^T \norm{\Psi^*(s)}_{\mathcal{L}(F^*,E^*)}^p \ds.
\end{aligned}
\end{equation*}
Since $\norm{\Psi^*(s)}_{\mathcal{L}(F^*,E^*)} = \norm{\Psi(s)}_{\mathcal{L}(E,F)} \leq \pi_p(\Psi(s))$
according to \cite[page 31]{Diestel}, it follows that
\begin{equation}
\label{estimate_of_drift_integral}
\E \left[\norm{I_B(\Psi)}^p\right]
\leq T^{p/q} \norm{B(1)}_{\mathcal{L}(E^*,\R)}^p \E \left[ \int_0^T \pi_p(\Psi(s))^p \ds\right].
\end{equation}
For estimating the second term in \eqref{decomposition_of_integral}, define the Banach space 
$$R_p = \bigg\{ X\colon (0,T] \times \Omega \to \R : \text{ measurable and } \sup\limits_{t \in (0,T]} \frac{1}{t^{1/p}} \big( \E \left[ \abs{X(t)}^p \right]\big)^{1/p} <\infty \bigg\}$$
with the norm $\norm{X}_{R_p} = \sup_{t \in (0,T]} \frac{1}{t^{1/p}} \left( \E \left[ \abs{X(t)}^p  \right] \right)^{1/p}$.
By standard properties of the real valued L\'evy martingales, see e.g.\ \cite[Th.\ 8.23(i)]{Peszat_Zabczyk}, it follows that there exists a constant $c>0$ such that
\begin{equation}
\label{continuity_integral_PRM}
\E \big[ \abs{M(t)x^*}^p \big] \leq ct \int_\R \abs{\beta}^p \,(\nu \circ (x^*)^{-1}) (\ud \beta)
\qquad\text{for all }x^\ast\in E^\ast.
\end{equation}
Here, we use that the  L\'evy measure of $M(1)x^*$ is given by $\nu \circ (x^*)^{-1}$. It follows that we can consider the map $M\colon E^{\ast}\to R_p$ defined by $Mx^\ast=(M(t)x^\ast:\, t\in (0,T])$. To show that $M$ is continuous, let $x_n^*$ converges to $x^*$ in $E^*$ and $Mx_n^*$ to some $Y$ in $R_p$. It follows that $M(t)x_n^* \to Y(t)$ in $L^p(\Omega;\R)$  for every $t\in (0,T]$. On the other hand, continuity of $M(t)\colon E^\ast \to L^1(\Omega,\mathcal{F},P;\R)$ implies $M(t)x_n^* \to M(t)x^*$ in $L^0(\Omega;\R)$. Thus, $Y(t)=M(t)x^*$  for all $t\in (0,T]$ a.s., and the closed graph theorem satisfies that $M\colon E^{\ast}\to R_p$ is continuous. It follows that 
\begin{equation}
\label{estimate_norm_cyl_increment}
\norm{M(t_{k+1}-t_k)}_{L(E^*;L^p(\Omega;\R))}^p \leq (t_{k+1}-t_k) \norm{M}_{\mathcal{L}(E^{\ast},R_p)}^p. 
\end{equation}
Let $J_{t_k,t_{k+1}}$ be the operators defined in \eqref{eq.def-J} with $L$ replaced by $M$. Since $F$ is of martingale type $p$ here exists a constant $C_p>0$ such that Lemma \ref{lem_construction_operators_J} and inequality \eqref{estimate_norm_cyl_increment} imply 
\begin{align*}
\E \left[ \norm{I_M(\Psi)}^p \right]
&= \E \left[ \norm{\sum_{k=1}^{N-1} J_{t_k,t_{k+1}}(\Psi_k)}^p \right] \\
&\leq C_p \E \left[ \sum_{k=1}^{N-1} \norm{J_{t_k,t_{k+1}}(\Psi_k)}^p \right]\\
&\leq C_p \sum_{k=1}^{N-1} \norm{M(t_{k+1}-t_k)}_{\mathcal{L}(E^*;L^p(\Omega;\R))}^p \E \left[ \pi_p(\Psi_k)^p \right]\\
&\leq C_p  \norm{M}_{\mathcal{L}(E^{\ast},R_p)}^p \E \left[ \int_0^T \pi_p(\Psi(s))^p \ds \right].
\end{align*}
Together with \eqref{estimate_of_drift_integral}, this completes the proof. 
\end{proof}

By rewriting Condition \eqref{finiteness_p_integral} as 
\begin{align*}
\int_{B_{\R}}\abs{\beta}^p\, (\nu\circ (x^{\ast})^{-1})(\ud\beta) < \infty
\quad\text{and}\quad
\int_{B_{\R}^c}\abs{\beta}^p\, (\nu\circ (x^{\ast})^{-1})(\ud\beta) <\infty
\quad\text{for all }x^\ast\in E^\ast, 
\end{align*}
it follows that  Condition \eqref{finiteness_p_integral} is equivalent to 
\begin{align*}
(L(t)x^* : t\ge 0) \text{ is $p$-integrable and has finite $p$-variation for each }
x^\ast\in E^\ast.
\end{align*}
This is a natural requirement if we want to control the moments, see \cite{Luschgy_Pages,Saint_Loubert_Bie} and Remark \ref{rem_stable_case} below.
Condition \eqref{finiteness_p_integral} implies in particular that the the Blumenthal-Getoor index of $(L(t)x^* : t\ge 0)$ is at most $p$. The interplay between the integrability of the L\'evy process and its Blumenthal-Getoor index was observed also in \cite{Chong_levy,Chong}.

\begin{Exa}[Gaussian case]
Note that if $p<2$, then $L$ cannot have the Gaussian part for the assertion to hold. Indeed, let $W$ be a one-dimensional Brownian motion and suppose for contradiction that
\begin{equation}
\label{continuity_of_Wiener_integral}
\E \left[ \abs{\int_0^T \Psi(t) \, \ud W(t)}^p \right] \leq C \E \left[ \int_0^T \abs{\Psi(t)}^p \dt \right]
\end{equation}
for some constant $C$ and every real-valued predictable process $\Psi$ with $\E \left[ \int_0^T \abs{\Psi(t)}^2 \dt \right] < \infty$.
Choose for each $n\in\N$ the stochastic process $\Psi_n(t) = \1_{[0,1/n]}(t)$ for $t \in[0,T]$. 
By \cite[Sec.\ 3.478]{Gradshteyn_Ryzhik} we calculate
$$\E \left[ \abs{\int_0^T \Psi_n(t) \, \ud W(t)}^p \right] 
= \E \left[ \abs{W\left(\frac1n\right)}^p \right] 
= \left(\frac1n \right)^{\frac{p}{2}} \frac{2^{\frac{p}{2}} \Gamma\left( \frac{p+1}{2}\right)}{\sqrt{\pi}}.$$
But on the other side, since $\E \left[ \int_0^T \abs{\Psi_n(t)}^p \dt \right] = \frac1n$,  
solving \eqref{continuity_of_Wiener_integral} for $n$ yields  
$$n^{1-\frac{p}{2}} \leq \frac{C\sqrt{\pi}}{2^\frac{p}{2} \Gamma\left(\frac{p+1}{2}\right)},$$
which results in a contradiction by taking the limit as $n\to \infty$.
\end{Exa}

\begin{Exa}[Stable case]
\label{rem_stable_case}
Let $E=L^{p^\prime}({\mathcal O})$ for $p^\prime=p/(p-1)$ and some ${\mathcal O}\subseteq \R^d$. 
The canonical $\alpha$-stable cylindrical L\'evy process has the characteristic function $\phi_{L(1)}(x^\ast)=\exp(-\norm{x^\ast}^\alpha)$ for each $x^\ast\in E^\ast$; see \cite{Riedle_stable}. It follows that the real-valued L\'evy process $(L(t)x^\ast:t\ge 0)$ is symmetric $\alpha$-stable with L\'evy measure $\nu\circ (x^\ast)^{-1}(\mathrm{d}\beta)=c\frac1{\abs{\beta}^{1+\alpha}} \mathrm{d}\beta$ for a constant $c>0$. 
Condition \eqref{finiteness_p_integral} fails to hold since 
$$\int_{B_\R} \abs{\beta}^p \, \nu \circ (x^*)^{-1} ( \mathrm{d}\beta) = \infty \, \text{ for } p\leq \alpha, \qquad \int_{B_\R^c} \abs{\beta}^p \, \nu \circ (x^*)^{-1} (\mathrm{d} \beta)  = \infty \, \text{ for } p\geq \alpha.$$
One can observe in a similar way as in the Gaussian case that the stochastic integral operator with respect to the $\alpha$-stable cylindrical L\'evy process $L$  is not continuous. If $\Psi_n(t) = \1_{[0,1/n]}(t)$, then in the inequality
\begin{equation}
\label{continuity_of_stable_integral}
\E \left[ \abs{\int_0^T \Psi_n(t) \, \dL(t)}^p \right] \leq C \E \left[ \int_0^T \abs{\Psi_n(t)}^p \dt \right],
\end{equation}
the left-hand side is infinite for $p \geq \alpha$. For $p<\alpha$ we use the self-similarity of the stable processes to calculate
$$\E \left[ \abs{\int_0^T \Psi_n(t) \, \dL(t)}^p \right] 
= \E \left[ \abs{L\left( \frac{1}{n} \right)}^p \right]
= \E \left[ \frac{1}{n^{p/\alpha}} \abs{L(1)}^p \right].$$
Solving  \eqref{continuity_of_stable_integral} for $n$ yields
$$n^{(\alpha-p)/\alpha} \leq \frac{C}{\E\left[ \abs{L(1)}^p \right]},$$
which results in a contradiction by taking the limit as $n\to \infty$.

Therefore, the stochastic integral mapping with respect to the $\alpha$-stable process cannot be continuous as a mapping from $L^p([0,T]\times \Omega,\mathcal{P},\ud t \otimes P;\R)$ to $L^p(\Omega,\mathcal{F}_T,P;\R)$ for any $p>0$. A moment inequality with different powers on the left and right-hand sides was proven in the case of real-valued integrands and vector-valued integrators in \cite{Rosinski_Woyczynski_moment}. They prove for any $\alpha$-stable L\'evy process $L$ and 
$p<\alpha$ that  
$$\E \left[ \bigg( \sup_{t\leq T} \norm{\int_0^t \Psi(s) \dL(s)} \bigg)^p \right] 
\leq C \E \left[ \left( \int_0^T \abs{\Psi(s)}^\alpha \dt \right)^{p/\alpha} \right].$$
\end{Exa}



\begin{Exa}
In various publications, e.g. \cite{Liu_Zhai_first, Peszat_Zabczyk_time_regularity, Priola_Zabczyk_2nd, Riedle_OU}, specific examples of the following kind of a cylindrical L\'evy process has been studied: let $E$ be a Hilbert space with an orthonormal basis $(e_k)$ and let $L$ be given by 
\begin{equation}
\label{diagonal_cylindrical_Levy}
L(t)x = \sum_{k=1}^\infty \langle x,e_k \rangle \ell_k(t) \qquad \text{for all } x \in E,
\end{equation}
where  $(\ell_k)$ is a sequence of independent, one-dimensional L\'evy processes
$\ell_k$ with characteristics $(b_k,0,\rho_k)$. Precise conditions under which the sum converges and related results can be found in \cite{Riedle_OU}. In this case, we claim that Condition 
\eqref{finiteness_p_integral} is satisfied if and only if
$$\sum_{k=1}^\infty \left( \int_\R \abs{\beta}^p \, \rho_k(\ud \beta) \right)^{\frac{2}{2-p}} < \infty.$$
It is shown in \cite[Lem.\ 4.2]{Riedle_OU} that the cylindrical L\'evy measure $\nu$ of $L$ is given by
\begin{align*}
\nu(A) = \sum_{k=1}^\infty \rho_k \circ m_{e_k}^{-1}(A) \qquad \text{for all } A\in  \mathcal{Z}(E),
\end{align*}
where $m_{e_k} \colon \R \to E$ is given by $m_{e_k}(x) = xe_k$.
Condition \eqref{finiteness_p_integral} simplifies to
$$\int_E \abs{\langle y,x \rangle}^p \, \nu(\ud x)
= \sum_{k=1}^\infty \int_E \abs{\langle y,x \rangle}^p \, (\rho_k \circ m_{e_k}^{-1})(\ud x)
= \sum_{k=1}^\infty \abs{\langle y,e_k \rangle}^p \int_\R \abs{\beta}^p \, \rho_k(\ud \beta) 
<\infty$$
for any $y \in E$. This is equivalent to  
$$\sum_{k=1}^\infty \alpha_k \int_\R \abs{\beta}^p \, \rho_k(\ud \beta) 
<\infty \qquad \text{ for any } (\alpha_k) \in \ell^{2/p}(\R_+),$$
which results in $\left( \int_\R \abs{\beta}^p \, \rho_k(\ud \beta) \right)_{k \in \N} \in \ell^{2/p}(\R)^* = \ell^{2/(2-p)}(\R)$.
\color{black}
\end{Exa}

\begin{Exa}
Another example are \emph{cylindrical compound Poisson} process, see e.g.\ 
\cite[Ex.\ 3.5]{Applebaum_Riedle}. These are cylindrical L\'evy processes of the form 
$$L(t)x^* = \sum_{k=1}^{N(t)} Y_kx^* \qquad\text{for all } x^*\in E^*,$$
where $N$ is a real-valued Poisson process with intensity $\lambda$ and $Y_k$ are identically distributed, cylindrical random variables, independent of $N$, and 
say with cylindrical distribution $\rho$. Since the L\'evy measure of $(L(t)x^\ast:t \ge 0)$ is given by 
$\lambda \rho\circ (x^\ast)^{-1}$, it follows that Condition   \eqref{finiteness_p_integral} is satisfied
if and only if 
\begin{equation}
\label{Levy_measure_p_integrable}
\int_E \abs{x^*(x)}^p \, \rho(\ud x) < \infty  \qquad x^* \in E^*.
\end{equation}

\end{Exa}

\begin{Rem}
If $p=2$ and $E$ and $F$ are Hilbert spaces, the space of admissible integrands $\Lambda(\Pi_2(E,F))$ are given by
$$\left\{ \Psi\colon [0,T] \times \Omega \to L_{\rm HS}(E,F) : \Psi \text{ is predictable and } \E \left[ \int_0^T \norm{\Psi(s)}_{L_{\rm HS}(E,F)}^2 \ds \right] < \infty \right\},$$
as in the work \cite{Riedle_L2}. This is only suboptimal, as it is known that in this case the space of admissible integrands can be enlarged to predictable processes satisfying $$\E \left[ {\displaystyle \int_0^T} \norm{\Psi(s)Q^{1/2}}_{L_{\rm HS}(E,F)}^2 \ds \right] < \infty,$$ where $Q$ is the covariance operator associated to $L$; see \cite{Tomasz-PhD}. In this way, the space of integrands depends on the L\'evy triplet of the integrator. One can ask if a similar result is possible in our more general setting for $p<2$ and for Banach spaces by replacing the covariance operator by the quadratic variation. 
\end{Rem}

\section{Existence and uniqueness of solution}
\label{sec_existence}

In this section we apply the developed integration theory to derive the existence of an evolution equation in a Banach space under standard assumptions. For this purpose, we consider
\begin{equation}
\label{SPDE}
\begin{aligned}
\dX(t) &= \big(AX(t) + B(X(t))\big)\dt + G\big(X(t)\big)\dL(t), \\
X(0) &= X_0,
\end{aligned}
\end{equation}
where $X_0$ is an $\mathcal{F}_0$-measurable random variable in a Banach space $F$ and the driving noise $L$ is a cylindrical L\'evy process in a Banach space $E$ with finite $p$-th  weak moments and finite $p$-variation. The operator $A$ is  the generator of a $C_0$-semigroup on $F$ and $B\colon F\to F$ and $G\colon F \to \Pi_p(E,F)$ are some functions. 

\begin{Def}
A mild solution of \eqref{SPDE} is a predictable process $X$ such that 
\begin{equation}
\label{integrability_condition_in_definition_of_solution}
\sup_{t\in [0,T]} \E \left[ \norm{X(t)}^p \right] < \infty
\end{equation}
for some $p\geq 1$, and such that, for all $t\in[0,T]$, we have $P$-a.s.
$$X(t) = S(t) X_0 + \int_0^t S(t-s) B(X(s)) \ds + \int_0^t S(t-s) G(X(s)) \dL(s).$$
\end{Def}

We assume Lipschitz and linear growth condition on the coefficients $F$ and $G$ and an integrability assumption on the initial condition.
\begin{Ass}
For fixed $p\in [1,2]$ we assume: 
\begin{enumerate}
\item[(A1)] there exist a function $b\in L^1([0,T];\R)$ such that for any $x,x_1,x_2 \in F$
and $t\in [0,T]$:
\begin{align*}
\norm{S(t)B(x)} & \leq b(t)(1+\norm{x}),\\
\norm{S(t)(B(x_1)-B(x_2))} & \leq b(t)\norm{x_1-x_2}.
\end{align*}
\item[(A2)] there exist a function $g\in L^p([0,T];\R)$ such that for any $x,x_1,x_2 \in F$
and $t\in [0,T]$:
\begin{align*}
\pi_p\big(S(t)G(x)\big) &\leq g(t)(1+\norm{x}), \\
\pi_p\big(S(t)(G(x_1)-G(x_2))\big) &\leq g(t)\norm{x_1-x_2}.
\end{align*}
\item[(A3)] $X_0 \in L^p(\Omega, \mathcal{F}_0,P;F)$.
\end{enumerate}
\end{Ass}

\begin{Th}
\label{th_existence_and_uniqueness_of_solution}
Let $p\in [1,2]$ and suppose that the Banach spaces $E$ and $F$ satisfy that
\begin{enumerate}
\item[{\rm (a)}]  $E$ is reflexive or has separable dual, $E^{**}$ has the approximation property, 
\item[{\rm (b)}] $F$ is of martingale type $p$,
\item[{\rm (c)}] if $p=1$, then $F$ has the Radon-Nikodym property.
\end{enumerate}
If  $L$ is a cylindrical L\'evy process such that \eqref{finiteness_p_integral} holds with some $p\in [1,2]$,
then conditions {\rm(A1)-(A3)} imply that there exists a unique mild solution of \eqref{SPDE}.
\end{Th}

\begin{proof}
We define the space
$$\tilde{\mathcal{H}}_{T} := \bigg\{ X\colon [0,T]\times \Omega \to F \text{ is predictable and } \sup_{t \in [0,T]} \E \left[ \norm{X(t)}^p \right] < \infty \bigg\},$$
and a family of seminorms for $\beta \geq 0$:
$$\norm{X}_{T,\beta} := \left( \sup_{t \in [0,T]} e^{-\beta t} \E \left[ \norm{X(t)}^p \right] \right)^{1/p}.$$
Let $\mathcal{H}_T$ be the set of equivalence classes of elements $\tilde{\mathcal{H}}_T$ relative to $\norm{\cdot}_{T,0}$. Define an operator $K \colon \mathcal{H}_T \to \mathcal{H}_T$ by $K(X) := K_0(X)+ K_1(X) + K_2(X)$,
where 
\begin{align*}
 K_0(X)(t) &:= S(t)X_0, \\
K_1(X)(t) &:= \int_0^t S(t-s) B(X(s)) \ds, \\
K_2(X)(t) &:= \int_0^t S(t-s) G(X(s)) \dL(s).
\end{align*}
The Bochner integral and the stochastic integral above are well defined because $X$ is predictable and for every $t\in[0,T]$ the mappings
$$[0,t]\times F \ni (s,x) \mapsto S(t-s)B(x), \qquad [0,t]\times F \ni (s,h) \mapsto S(t-s)G(x)$$
are continuous. The appropriate integrability condition follows from \eqref{finitness_of_K_1} and \eqref{finitness_of_K_2} below.

For applying Banach's fixed point theorem, we first show that  $K$ indeed maps to $\mathcal{H}_T$. 
Choose constants $m\geq 1$ and $\omega\in \R$ such that $\norm{S(t)} \leq m e^{\omega t}$ for each $t\geq 0$. It follows that
$$\sup_{t \in [0,T]} \E \left[ \norm{S(t) X_0}^p \right] \leq m e^{\abs{\omega} T} \E \left[ \norm{X_0}^p \right] < \infty.$$
By Assumption (A1) and H{\"o}lder's inequality, we obtain with $q=\frac{p}{p-1}$ that
\begin{equation}
\label{finitness_of_K_1}
\begin{aligned}
&\sup_{t \in [0,T]} \E \left[ \norm{\int_0^t S(t-s)B(X(s)) \ds}^p \right] \\
&\qquad \leq \sup_{t \in [0,T]} \E \left[ \left(\int_0^t b(t-s)(1+\norm{X(s)}) \ds\right)^p \right]\\
&\qquad\leq \sup_{t \in [0,T]} \E \left[  \left(\int_0^t b(t-s) ds \right)^{p/q}\int_0^t b(t-s)(1+\norm{X(s)})^p \ds \right] \\
&\qquad\leq \left( \int_0^T b(s) \ds \right)^{p/q} 2^{p-1} (1+ \norm{X}_{T,0}) \sup_{t \in [0,T]} \int_0^t b(t-s) \ds \\
&\qquad= \left( \int_0^T b(s) \ds \right)^{1+p/q} 2^{p-1} (1+ \norm{X}_{T,0})  \\
&\qquad <\infty.
\end{aligned}
\end{equation} 
Similarly, we conclude from Assumption (A2) and Theorem \ref{th_continuity_of_integral} that 
there exists a constant $c>0$ such that
\begin{equation}
\label{finitness_of_K_2}
\begin{aligned}
&\sup_{t \in [0,T]} \E \left[ \norm{\int_0^t S(t-s)G(X(s)) \dL(s)}^p \right]\\
&\qquad \leq c \sup_{t \in [0,T]} \E \left[ \int_0^t \pi_p(S(t-s)G(X(s)))^p \ds \right] \\
&\qquad \leq c \sup_{t \in [0,T]} \E \left[ \int_0^t g(t-s)^p(1+\norm{X(s)})^p \ds \right] \\
&\qquad\leq c 2^{p-1}(1+\norm{X}_{T,0}^p) \int_0^T g(s)^p \ds \\
&\qquad < \infty.
\end{aligned}
\end{equation} 
Next, we establish that $K$ is stochastically continuous. For this purpose, let $\epsilon>0$. 
For each $t\ge 0$ we obtain 
\begin{align*}
&\E \left[ \norm{K_1(t+\epsilon)-K_1(t)} \right] \\
&= \E \left[ \norm{\int_0^{t+\epsilon} S(t+\epsilon-s)B(X(s)) \ds - \int_0^t S(t-s)B(X(s)) \ds} \right] \\
&= \E \left[ \norm{\int_t^{t+\epsilon} S(t+\epsilon-s)B(X(s)) \ds + \int_0^t (S(\epsilon)-\mathrm{Id})S(t-s)B(X(s)) \ds} \right] \\
&\leq \E \left[ \int_t^{t+\epsilon} \norm{S(t+\epsilon-s)B(X(s))} \ds + \int_0^t \norm{(S(\epsilon)-\mathrm{Id})S(t-s)B(X(s))} \ds \right] \\
&=: I_1 + I_2.
\end{align*}
Since $\norm{X(s)} \leq 1+ \norm{X(s)}^p $ for all $s\ge 0$, it follows, for $\epsilon \to 0$, that
$$I_1
\leq \E \left[ \int_t^{t+\epsilon} b(t+\epsilon-s) (1+\norm{X(s)}) \ds \right]
\leq (2+\norm{X(s)}^p) \int_0^\epsilon b(s) \ds
\to 0.$$
With the same estimate $\norm{X(s)} \leq 1+ \norm{X(s)}^p $ we obtain 
\begin{align*}
\norm{(S(\epsilon)-\mathrm{Id})S(t-s)B(X(s))}
&\leq (1+me^{\abs{\omega}})b(t-s)(1+\norm{X(s)}) \\
&\leq (2+\norm{X}_{T,0}^p)(1+me^{\abs{\omega}}) b(t-s). 
\end{align*}
Since the integrand of $I_2$  tends to $0$ as $\epsilon \to 0$ by the strong continuity of the semigroup, Lebesgue's dominated convergence theorem shows that $I_2$ tends to $0$ as $\epsilon \to 0$. 

For $K_2$ we obtain by  Theorem \ref{th_continuity_of_integral} that  there exists a constant $c>0$ such that
\begin{align*}
&\E \left[ \norm{K_2(t+\epsilon)-K_2(t)}^p \right] \\
&= \E \left[ \norm{\int_0^{t+\epsilon} S(t+\epsilon-s)G(X(s)) \dL(s) - \int_0^t S(t-s)G(X(s)) \dL(s)}^p \right] \\
&= \E \left[ \norm{\int_t^{t+\epsilon} S(t+\epsilon-s)G(X(s)) \dL(s) + \int_0^t (S(\epsilon)-\mathrm{Id})S(t-s)G(X(s)) \dL(s)}^p \right] \\
&\leq 2^{p-1} \E \left[ \norm{\int_t^{t+\epsilon} S(t+\epsilon-s)G(X(s)) \dL(s)}^p + \norm{\int_0^t (S(\epsilon)-\mathrm{Id})S(t-s)G(X(s)) \dL(s)}^p \right] \\
&\leq c 2^{p-1} \E \left[ \int_t^{t+\epsilon} \pi_p(S(t+\epsilon-s)G(X(s)))^p \ds + \int_0^t \pi_p((S(\epsilon)-\mathrm{Id})S(t-s)G(X(s)))^p \ds \right] \\
&\leq c 2^{p-1} \E \left[ \int_t^{t+\epsilon} 2^{p-1} g(t+\epsilon-s)^p (1+\norm{X(s)}^p) \ds + \int_0^t \pi_p((S(\epsilon)-\mathrm{Id})S(t-s)G(X(s)))^p \ds \right] \\
&=: c 2^{p-1}(2^{p-1}J_1 + J_2), 
\end{align*}
where
$$J_1 
= \E \left[ \int_t^{t+\epsilon} g(t+\epsilon-s)^p (1+\norm{X(s)}^p) \ds \right]
\leq (1+\norm{X}_{T,0}^p) \int_t^{t+\epsilon} g(t+\epsilon-s)^p \ds
\to 0$$
as $\epsilon \to 0$, and
$$J_2
= \E \left[ \int_0^t \pi_p((\Id-S(\epsilon))S(t-s)G(X(s)))^p \ds \right].$$
By Theorem \ref{th_convergence_in_p_summing_norm} the integrand $\pi_p((\Id-S(\epsilon))S(t-s)G(X(s)))^p$ converges to $0$ for all $t$ and $\omega\in \Omega$. 
Moreover it is bounded by $(1+me^{\abs{\omega}})^p g(t-s)^p(1+\norm{X(s)})^p$, which is $\ud t\otimes P$-integrable. Thus, Lebesgue's theorem on dominated convergence implies that $J_2\to 0$ as $\epsilon \to 0$ which completes the proof of stochastic continuity of $K$. In particular, stochastic continuity guarantees the existence of a predicable modification of $K$ by \cite[Prop.\ 3.21]{Peszat_Zabczyk}. 
In summary, we obtain that $K$ maps $\mathcal{H}_T$  to $\mathcal{H}_T$. 

For applying Banach's fixed point theorem it is enough to show that $K$ is a contraction for some $\beta$. We have
$$\norm{K(X_1) - K(X_2)}_{T,\beta}^p
\leq 2^{p-1} \left( \norm{K_1(X_1)-K_1(X_2)}_{T,\beta} + \norm{K_2(X_1)-K_2(X_2)}_{T,\beta} \right).$$
For the part corresponding to the drift we calculate similarly to \cite[Th.\ 9.29]{Peszat_Zabczyk}
\begin{align*}
&\norm{K_1(X_1)-K_1(X_2)}_{T,\beta}^p \\
&\leq \sup_{t \in [0,T]} e^{-\beta t} \E \left[ \left( \int_0^t b(t-s) \norm{X_1(s)-X_2(s)} \ds \right)^p \right] \\
&= \sup_{t \in [0,T]} e^{-\beta t} \E \left[ \left( \int_0^t b(t-s)^{1/q} b(t-s)^{1/p} \norm{X_1(s))-X_2(s)} \ds \right)^p \right] \\
&\leq \left(\int_0^T b(t-s) \ds \right)^{p/q} \sup_{t \in [0,T]} e^{-\beta t} \int_0^t b(t-s) \E \left[ \norm{X_1(s))-X_2(s)}^p \right] \ds  \\
&= \left(\int_0^T b(s) \ds \right)^{p/q} \sup_{t \in [0,T]} e^{-\beta t} \int_0^t b(t-s)e^{\beta s} e^{-\beta s}\E \left[ \norm{X_1(s)-X_2(s)}^p \right] \ds \\
&\leq \left(\int_0^T b(s) \ds \right)^{p/q} \norm{X_1-X_2}_{T,\beta}^p \sup_{t \in [0,T]}  \int_0^t b(t-s)e^{-\beta(t-s)} \ds \\
&= C(\beta) \norm{X_1-X_2}_{T,\beta}^p
\end{align*}
with $C(\beta)= \left(\int_0^T b(s) \ds \right)^{p/q} \int_0^T b(s)e^{-\beta s} \ds \to 0$ as $\beta \to \infty$.

In the following calculation for the part corresponding to the diffusion we use in the first inequality the continuity of the stochastic integral formulated in Theorem \ref{th_continuity_of_integral}:
\begin{align*}
\norm{K_2(X_1)-K_2(X_2)}_{T,\beta}^p
&\leq c \sup_{t \in [0,T]} e^{-\beta t} \E \left[ \int_0^t \pi_p( S(t-s)(G(X_1(s))-G(X_2(s))))^p \ds \right] \\
&\leq c \sup_{t \in [0,T]} e^{-\beta t} \E \left[ \int_0^t g(t-s)^p \norm{X_1(s)-X_2(s)}^p \ds \right] \\
&= c \sup_{t \in [0,T]} e^{-\beta t} \E \left[ \int_0^t g(t-s)^p e^{\beta s} e^{-\beta s}\norm{X_1(s)-X_2(s)}^p \ds \right] \\
&\leq c \norm{X_1-X_2}_{T,\beta}^p \sup_{t \in [0,T]}   \int_0^t e^{-\beta (t-s)} g(t-s)^p \ds \\
&= C'(\beta) \norm{X_1-X_2}_{T,\beta}^p,
\end{align*}
where $C'(\beta) = c \int_0^T e^{-\beta s} g(s)^p \ds \to 0$ as $\beta \to \infty$.
Consequently, Banach's fixed point theorem implies that there exists a unique $X\in \mathcal{H}_T$
such that $K(X)=X$ which completes the proof. 
\end{proof}

\begin{Rem}
Note that if $E$ and $F$ are Hilbert spaces, then they satisfy assumption (ii) in Theorem \ref{th_continuity_of_integral}, see e.g.\ \cite[Cor.\ 1, p.\ 109]{Schaefer_Wolff}. Thus for $p=2$ we recover \cite{Riedle_L2}.
\end{Rem}

\begin{Rem}
For processes of the form \eqref{diagonal_cylindrical_Levy} the integrability assumption can be relaxed to include for example stable processes in the same way as in \cite{Kosmala_Riedle} where the existence of variational solutions  is demonstrated. The details can be found in the PhD thesis of the first author.
\end{Rem}

\subsection*{Acknowledgement}
I would like to thank Przemys\l{}aw Wojtaszczyk for a helpful discussion
and Emilio Fedele for suggesting to use the finite rank operators in the proof of Theorem \ref{th_convergence_in_p_summing_norm}.

\bibliographystyle{plain}

\end{document}